\documentclass[10pt,a4paper]{amsart}

\usepackage{pst-plot}
   \usepackage{pstricks}
   \usepackage{multido}
   \usepackage{pstricks-add} 

\usepackage{amsmath}
\usepackage{amsfonts}
\usepackage{color} 

\newtheorem{theo}{Theorem}[section] 
\newtheorem{prop}[theo]{Proposition}
\newtheorem{lem}[theo]{Lemma}

\begin{document}

\date{\today}

\title[Abelian  varieties]{The characteristic polynomials of abelian varieties of dimensions 3 over finite fields}
\author{Safia Haloui}
\address{Institut de Math\'ematiques de Luminy, Marseille, France}

\email{haloui@iml.univ-mrs.fr}



\begin{abstract}
We describe the set  of characteristic polynomials of abelian varieties of dimension $3$ over finite fields.
\end{abstract}

\maketitle
\section{Introduction and results}
The isogeny class of an abelian variety over a finite field is determined by its characteristic polynomial (i.e. the characteristic polynomial of its  Frobenius endomorphism). We describe the set  of characteristic polynomials which occur in dimension $3$; this completes the work of Xing \cite{xing} (we will recall his results in this section). Since the problem has been solved in dimensions $1$ and $2$ (see \cite{water}, \cite{ruck} and \cite{mana}), it is sufficient to focus on simple abelian varieties.

Let $p(t)$ be the characteristic polynomial of an abelian variety of dimension $g$ over $\mathbb{F}_q$ (with $q=p^n$). Then the set of its roots has the form $\{\omega_1,\overline{\omega_1},\dots ,\omega_g, \overline{\omega_g}\}$ where the $\omega_i$'s  are $q$-Weil numbers. A monic polynomial with integer coefficients which satisfies this condition is called a \textit{Weil polynomial}. Thus every Weil polynomial of degree $3$ has the form
$$p(t)=t^6+a_1t^5+a_2t^4+a_3t^3+qa_2t^2+q^2a_1+q^3$$
for certains integers $a_1$, $a_2$ and $a_3$. The converse is false; indeed, since the absolute value of the roots of $p(t)$ is prescribed (equal to $\sqrt q$), its coefficients have to be bounded. Section 2 is dedicated to the proof of the following proposition:
\begin{theo}\label{coeffweil}
Let $p(t)=t^6+a_1t^5+a_2t^4+a_3t^3+qa_2t^2+q^2a_1t+q^3$ be a polynomial with integer coefficients. Then $p(t)$ is a Weil polynomial if and only if
  either
$$f(t)=(t^2-q)^2(t^2+\beta t+q)$$
where $\beta\in\mathbb{Z}$ and  $\vert\beta\vert <2\sqrt{q}$, or
 the following conditions hold
\begin{enumerate}
\item $\vert a_1\vert< 6\sqrt q$,
\item $4\sqrt q\vert a_1\vert -9q< a_2\leq\frac{a_1^2}{3}+3q$,
\item $-\frac{2a_1^3}{27}+\frac{a_1a_2}{3}+qa_1-\frac{2}{27}(a_1^2-3a_2+9q)^{3/2}\leq a_3\leq -\frac{2a_1^3}{27}+\frac{a_1a_2}{3}+qa_1+\frac{2}{27}(a_1^2-3a_2+9q)^{3/2}$,
\item $-2qa_1-2\sqrt qa_2-2q\sqrt q< a_3< -2qa_1+2\sqrt qa_2+2q\sqrt q.$
\end{enumerate}
\end{theo}

The Honda-Tate Theorem gives us a bijection between the set of conjugacy classes of $q$-Weil numbers and the set of isogeny classes of simple abelian varieties over $\mathbb{F}_q$. Moreover, the characteristic polynomial of a simple abelian variety of dimension $3$ over $\mathbb{F}_q$ has the form $p(t)=h(t)^e$ where $h(t)$ is an irreducible Weil polynomial and $e$ is an integer. Obviously $e$ must divide $6$. 

As remarked by Xing  \cite{xing}, $e$ cannot be equal to $2$ or $6$, otherwise $p(t)$ would have a real root and real  $q$-Weil numbers ($\pm\sqrt{q}$) correspond to dimension $1$ or $2$ (according to the parity of $n$) abelian varieties  (see \cite{water}).

When $e=3$, using a result from Maisner and Nart \cite[Proposition 2.5]{mana}, we get the following proposition, proved by Xing, which  gives us the form of $h(t)$.
\begin{prop}[Xing]\label{polye3}
Let $\beta\in\mathbb{Z}$, $\vert\beta\vert <2\sqrt{q}$. There exists a simple abelian variety of dimension $3$ over $\mathbb{F}_q$ with $h(t)=t^2+\beta t+q$ if and only if $3$ divides $n$ and $\beta =aq^{1/3}$, where $a$ is an integer coprime with $p$.
\end{prop}
It follows that a simple abelian variety of dimension $3$ with a reducible characteristic polynomial has $p$-rank $0$; this fact was proved by  Gonz\'alez   \cite{gonz}. Note that the Newton polygon of a polynomial from Proposition \ref{polye3} is of type $1/3$ (see Figure \ref{type13}, Section 4).

It remains to see what happens when $p(t)$ is irreducible ($e=1$). First, we need an irreducibility criterion for Weil polynomials. In section 3, we prove the following proposition:
\begin{prop}\label{irred}
Set
$$r=-\frac{a_1^2}{3}+a_2-3q\quad\mbox{ and }\quad s=\frac{2a_1^3}{27}-\frac{a_1a_2}{3}-qa_1+a_3$$
and 
$$\Delta=s^2-\frac{4}{27}r^3\quad\mbox{ and }\quad u=\frac{-s+\sqrt{\Delta}}{2}\cdot$$
Then $p(t)$ is irreducible over $\mathbb{Q}$ if and only if $\Delta\neq 0$ and $u$ is not a cube in $\mathbb{Q}(\sqrt{\Delta})$.
\end{prop}

Next, we determine the possible Newton polygons for $p(t)$; this is the aim of Section 4.
\begin{theo}\label{newt} Let $p(t)=t^6+a_1t^5+a_2t^4+a_3t^3+qa_2t^2+q^2a_1+q^3$ be an irreducible Weil polynomial. Then $p(t)$ is the characteristic polynomial of an abelian variety of dimension $3$ if and only if one of the following conditions holds
\begin{enumerate}
\item $v_p(a_3)=0$,
\item $v_p(a_2)=0$, $v_p(a_3)\geq n/2$ and $p(t)$ has no root of valuation $n/2$ in $\mathbb{Q}_p$,
\item $v_p(a_1)=0$ $v_p(a_2)\geq n/2$, $v_p(a_3)\geq n$ and $p(t)$ has no root of valuation $n/2$ in $\mathbb{Q}_p$,
\item $v_p(a_1)\geq n/3$, $v_p(a_2)\geq 2n/3$, $v_p(a_3)=n$ and $p(t)$ has no root $\mathbb{Q}_p$,
\item $v_p(a_1)\geq n/2$, $v_p(a_2)\geq n$, $v_p(a_3)\geq 3n/2$ and $p(t)$ has no root nor factor of degree $3$ in $\mathbb{Q}_p$.
\end{enumerate}
The $p$-ranks of abelian varieties in cases 1, 2, 3, 4 and 5 are respectively  $3$, $2$, $1$, $0$ and $0$. The abelian varieties in case 5 are supersingular.
\end{theo}

It is possible to make condition (5) of Proposition \ref{newt} more explicit. Indeed, in \cite{nr}, Nart and Ritzenthaler gave the list of supersingular $q$-Weil numbers of degree $6$. We derive from it the following proposition (see Section 5).
\begin{prop}\label{supers} If $p(t)$ is the characteristic polynomial of a supersingular abelian variety of dimension $3$ then one of the following conditions holds
\begin{enumerate}
\item $(a_1,a_2,a_3)=(q^{1/2},q,q^{3/2})$ or $(-q^{1/2},q,-q^{3/2})$, $q$ is a square and $7\not\vert (p^3-1)$,
\item $(a_1,a_2,a_3)=(0,0,q^{3/2})$ or $(0,0,-q^{3/2})$, $q$ is a square and $3\not\vert (p-1)$, 
\item $(a_1,a_2,a_3)=(\sqrt{pq},3q,q\sqrt{pq})$ or $(-\sqrt{pq},3q,-q\sqrt{pq})$, $p=7$ and $q$ is not a square, 
\item $(a_1,a_2,a_3)=(0,0,q\sqrt{pq})$ or $(0,0,-q\sqrt{pq})$, $p=3$ and $q$ is not a square.
\end{enumerate}
\end{prop}


\section{The coefficients of Weil polynomials of degree $6$}
In order to prove Theorem \ref{coeffweil}, we use Robinson's method (described by Smyth in \cite[\S 2, Lemma]{smyth}). Fixing a polynomial of degree $3$ and doing an explicit calculation we get the following lemma:
\begin{lem}\label{coeff}
Let $f(t)=t^3+r_1t^2+r_2t+r_3$ be a monic polynomial of degree $3$ with real coefficients. Then $f(t)$ has all real positive roots if and only if the following conditions hold:
\begin{enumerate}
\item $r_1< 0$,
\item $0< r_2\leq \frac{r_1^2}{3}$,
\item $\frac{r_1r_2}{3}-\frac{2r_1^3}{27}-\frac{2}{27}(r_1^2-3r_2)^{3/2}\leq r_3\leq\frac{r_1r_2}{3}-\frac{2r_1^3}{27}+\frac{2}{27}(r_1^2-3r_2)^{3/2}$ and $r_3< 0$.
\end{enumerate}
\end{lem}
\begin{proof} If $f(t)$ has all real positive roots, so do all its derivates. Thus condition (1) is obvious. Let $f_0'(t)$ be the primitive of $f''(t)$ vanishing at $0$; if we add a constant to $f_0'(t)$ so that all its roots are real and positive, we obtain (2). Repeating this process with a primitive of $f'(t)$ vanishing at $0$, we obtain (3).
\end{proof}

Let $x=(x_1,x_2,x_3)\in\mathbb{C}^3$ and set
\begin{eqnarray} \label{polys}
p_x(t)=\prod_{i=1}^3(t^2+x_it+q)\mbox{, }
\end{eqnarray}
\begin{eqnarray*}
f_x(t)=\prod_{i=1}^3(t-(2\sqrt q+x_i))\quad\mbox{ and }\quad\widetilde{f}_x(t)=\prod_{i=1}^3(t-(2\sqrt q-x_i)).
\end{eqnarray*}

If $p_x(t)$ is a Weil polynomial (thus  $x_i=-(\omega_i+\overline{\omega_i})$, where $\omega_1,\overline{\omega_1},\dots ,\omega_g, \overline{\omega_g}$ are the roots of $p_x(t)$) then the roots of $f_x(t)$ and $\widetilde{f}_x(t)$ are real and positive. Conversely, suppose that the roots of $f_x(t)$ and $\widetilde{f}_x(t)$ are real and positive, then if $p_x(t)$ has integer coefficients and it is a Weil polynomial.

For $i=1,2,3$, let $a_i$ denote the coefficient associated to $p_x(t)$ in Proposition \ref{coeffweil}, $s_i$ the $i$th symmetric function of the $x_i$'s and $r_i$ and $\widetilde{r}_i$ the respective $i$th coefficients of $f_x(t)$ and $\widetilde{f}_x(t)$.

Expanding the expression of $p_x(t)$ in (\ref{polys}), we find
\begin{eqnarray*}
a_1 & = & s_1\\
a_2 & = & s_2+3q\\
a_3 & = & s_3+2qs_1.
\end{eqnarray*}
In the same way, expanding the expressions of $f_x(t)$ and $\widetilde{f}_x(t)$, we find
$$\begin{array}{lclclcl}
r_1 & = & -6\sqrt q-s_1                                      &                                          & \widetilde{r}_1 & = & -6\sqrt q+s_1\\
r_2 & = & 12q+4\sqrt qs_1+s_2                      & \quad\mbox{ and } \quad & \widetilde{r}_2 & = & 12q-4\sqrt qs_1+s_2\\
r_3 & = & -8q\sqrt q-4qs_1-2\sqrt qs_2-s_3 &                                  & \widetilde{r}_3 & = & -8q\sqrt q+4qs_1-2\sqrt qs_2+s_3.
\end{array}$$
Therefore we have
$$\begin{array}{lclclcl}
r_1 & = & -6\sqrt q-a_1                                      &                                          & \widetilde{r}_1 & = & -6\sqrt q+a_1\\
r_2 & = & 9q+4\sqrt qa_1+a_2                      & \quad\mbox{ and } \quad & \widetilde{r}_2 & = & 9q-4\sqrt qa_1+a_2\\
r_3 & = & -2q\sqrt q-2qa_1-2\sqrt qa_2-a_3 &                                  & \widetilde{r}_3 & = & -2q\sqrt q+2qa_1-2\sqrt qa_2+a_3.
\end{array}$$
The polynomials $f_x(t)$ and $\widetilde{f}_x(t)$ satisfy condition 1 of Lemma \ref{coeff} if and only if
$$\vert a_1\vert < 6\sqrt q.$$
The polynomials $f_x(t)$ and $\widetilde{f}_x(t)$ satisfy condition 2 of Lemma \ref{coeff} if and only if 
$$4\sqrt q\vert a_1\vert -9q < a_2\leq\frac{a_1^2}{3}+3q.$$
We find that the first inequality in condition 3 of Lemma \ref{coeff} holds for $f_x(t)$ if and only if it holds for $\widetilde{f}_x(t)$ if and only if
$$-\frac{2a_1^3}{27}+\frac{a_1a_2}{3}+qa_1-\frac{2}{27}(a_1^2-3a_2+9q)^{3/2}\leq a_3\leq -\frac{2a_1^3}{27}+\frac{a_1a_2}{3}+qa_1+\frac{2}{27}(a_1^2-3a_2+9q)^{3/2}.$$
Finally,  $f_x(t)$ and $\widetilde{f}_x(t)$ satisfy the second inequality in condition 3 of Lemma \ref{coeff} if and only if
$$-2qa_1-2\sqrt qa_2-2q\sqrt q < a_3< -2qa_1+2\sqrt qa_2+2q\sqrt q.$$
Hence Theorem \ref{coeffweil} is proved.


\section{Irreducible Weil polynomials}
Given a Weil polynomial $p(t)=\prod_{i=1}^g(t^2+x_it+q)$, we consider its real Weil polynomial $f(t)=\prod_{i=1}^g(t+x_i)$.
\begin{prop}Suppose that  $g\geq 2$ and $p(t)\neq (t-\sqrt{q})^2(t+\sqrt{q})^2$. Then $p(t)$ is irreducible over $\mathbb{Q}$ if and only if $f(t)$ is irreducible over $\mathbb{Q}$.\end{prop}
\begin{proof}
Suppose that $p(t)$ is reducible. It is sufficient to prove that $p(t)$ factors as the product of two Weil polynomials (then $f(t)$ will be the product of its associated polynomials). The polynomial $p(t)$ decomposes as $p(t)=(t-\sqrt{q})^{2k}(t+\sqrt{q})^{2\ell}h(t)$ where $h(t)$ has no real root. If $k\neq\ell$, $\sqrt{q}\in \mathbb{Q}$ and $p(t)$ factors obviously. The same conclusion holds when $k=\ell\neq 0$ and $h(t)\neq 1$. If $k=\ell >1$ and $h(t)=1$, we have the decomposition $p(t)=[(t-\sqrt{q})^{2}(t+\sqrt{q})^{2}][(t-\sqrt{q})^{2k-2}(t+\sqrt{q})^{2\ell -2}]$. Finally, if $k=\ell =0$, by hypothesis $h(t)$ is the product of two monic non-constant polynomials which are obviously Weil polynomials.

Conversely, if $f(t)$ is reducible, we can assume (possibly  changing labels of the $x_i$'s) that there exists an integer $k$ between $1$ and $(g-1)$ such that the polynomials $\prod_{i=1}^k(t+x_i)$ and $\prod_{i=k+1}^g(t+x_i)$ have integer coefficients. Thus $\prod_{i=1}^k(t^2+x_it+q)$ and $\prod_{i=k+1}^g(t^2+x_it+q)$ have integer coefficients and their product is $p(t)$.
\end{proof}

Now we focus on the case $g=3$. In order to know if $p(t)$ is irreducible, it is sufficient to check if $f(t)$ (a polynomial of degree $3$ with all real roots) is irreducible. To do this, we use Cardan's method. Let us recall quickly what it is. 

Fixing a polynomial $h(t)=t^3+rt+s$, we set $\Delta=s^2-\frac{4}{27}r^3$. If $h(t)$ has all real roots, we have $\Delta \leq 0$. Moreover, $\Delta =0$ if and only if $h(t)$ has a double root. When $\Delta <0$, we set $u=\frac{-s+\sqrt{\Delta}}{2}$. The roots of $h(t)$ are in the form $(v+\overline{v})$ where $v$ is a cube root of $u$.

We apply this to $f(t)=t^3+a_1t^2+(a_2-3q)t+(a_3-2qa_1)$:

\begin{proof}[Proof of proposition \ref{irred}]
We set $h(t)=t^3+rt+s$ so that $f(t)=h(t+\frac{a_1}{3})$. The polynomial $f(t)$ is reducible if and only if it has a root in $\mathbb{Q}$ if and only if $h(t)$ has a root in $\mathbb{Q}$.

If $\Delta =0$, $f(t)$ is reducible. Suppose that $\Delta <0$. 
If $u$ is the cube of a certain $v\in\mathbb{Q}(\sqrt{\Delta})$, we have obviously $(v+\overline{v})\in\mathbb{Q}$. 
Conversely, if $h(t)$ has a root in $\mathbb{Q}$ then $u$ has a cube root $v=a+ib$ with $a\in\mathbb{Q}$ and we have
$$u=v^3=(a^3-3ab^2)+ib(3a^2b-b^2).$$
If $a\neq 0$, identifying real parts in the last equality, we see that $b^2\in \mathbb{Q}$, then, identifying imaginary parts, $b\in\mathbb{Q}(\sqrt{-\Delta})$. Therefore $v\in\mathbb{Q}(\sqrt{\Delta})$.
If $a=0$, then $s=0$ and $\Delta =\frac{4}{27}r^3=(\frac{2}{3}r)^2\frac{r}{3}$. Thus $u=\frac{1}{2}\sqrt{\frac{4}{27}r^3}=(\sqrt{\frac{r}{3}})^3$ is a cube in $\mathbb{Q}(\sqrt{\Delta})=\mathbb{Q}(\sqrt{\frac{r}{3}})$.

\end{proof}

\section{Newton polygons}
Let $p(t)$ be an irreducible Weil polynomial of degree $3$ and $e$ the least common denominator of $v_p(f(0))/n$ where $f(t)$ runs through the irreducible factors of $p(t)$ over $\mathbb{Q}_p$ (the field of $p$-adic numbers). By \cite{milwat}, $p(t)^e$ is the characteristic polynomial of a simple abelian variety. Thus $p(t)$ is the characteristic polynomial of an abelian variety of dimension $3$ if and only if  $e$ is equal to $1$ that is, $v_p(f(0))/n$ are integers. One way to obtain information about  $p$-adic valuations of the roots of $p(t)$ is to study its Newton polygon (see \cite{weiss}). The condition "$v_p(f(0))/n$ are integers" implies that the projection onto the $x$-axis of an edge of the Newton polygon having a slope $\ell n/k$ (with $\mbox{pgcd}(\ell ,k)=1$) has length a multiple of $k$. We graph the Newton polygons satisfying this condition and in each case, we give a necessary and sufficient condition to have $e=1$. The obtained results are summarized in Theorem \ref{newt}.

\bigskip
\noindent
\textbf{Ordinary case: }$\mathbf{v_p(a_3)=0}$

The Newton polygon of $p(t)$ is represented in Figure \ref{ordinaire} and we always have $e=1$.
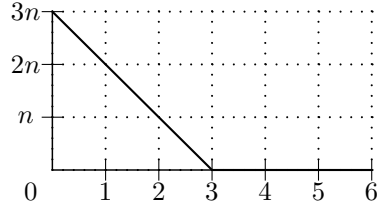
\begin{figure}[h]
   \centering
    \begin{pspicture}(0,0)(6,3)
       \psset{unit=7mm,labels=none}
  \rput(-0.4,-0.4){$0$}
  \rput(-0.5,1){$n$}
  \rput(-0.5,2){$2n$}
  \rput(-0.5,3){$3n$}
  \rput(1,-0.4){$1$}
    \rput(2,-0.4){$2$}
    \rput(3,-0.4){$3$}
   \rput(4,-0.4){$4$}
    \rput(5,-0.4){$5$}
    \rput(6,-0.4){$6$}
   \psaxes[linewidth=.7\pslinewidth]{-}(0,0)(0,0)(6,3)
  \psgrid[griddots=5, subgriddiv=0, gridlabels=0pt](0,0)(6,3)
   \psline(0,3)(3,0)(6,0)
\end{pspicture}
 \caption{\label{ordinaire}Ordinary case}
\end{figure}

\bigskip
\noindent
\textbf{p-rank 2 case: }$\mathbf{v_p(a_3)>0}$\textbf{ and }$\mathbf{v_p(a_2)=0}$

The only Newton polygon for which $e=1$ is represented in Figure \ref{prang2}.
\begin{figure}[h]
   \centering
    \begin{pspicture}(0,0)(6,3)
    \psset{unit=7mm,labels=none}
  \rput(-0.4,-0.4){$0$}
  \rput(-0.5,1){$n$}
  \rput(-0.5,2){$2n$}
  \rput(-0.5,3){$3n$}
  \rput(1,-0.4){$1$}
    \rput(2,-0.4){$2$}
    \rput(3,-0.4){$3$}
   \rput(4,-0.4){$4$}
    \rput(5,-0.4){$5$}
    \rput(6,-0.4){$6$}
   \psaxes[linewidth=.7\pslinewidth]{-}(0,0)(0,0)(6,3)
  \psgrid[griddots=5, subgriddiv=0, gridlabels=0pt](0,0)(6,3)
   \psline(0,3)(2,1)(4,0)(6,0)
\end{pspicture}
 \caption{\label{prang2}$p$-rank $2$ case}
\end{figure}
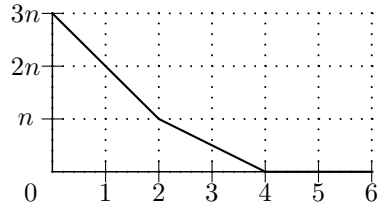

This is the Newton polygon of  $p(t)$ if and only if $v_p(a_3)\geq n/2$. If this condition holds, $p(t)$ has a factor  in $\mathbb{Q}_p$ of degree $2$ with roots of valuation $n/2$ and thus $e=1$ if and only if this factor is irreducible, that is, if and only if $p(t)$ has no root of valuation $n/2$ in $\mathbb{Q}_p$ (note that when $n$ is odd, this last condition always holds).

\bigskip
\noindent
\textbf{p-rank 1 case: }$\mathbf{v_p(a_3)>0}$\textbf{, }$\mathbf{v_p(a_2)>0}$\textbf{ and }$\mathbf{v_p(a_1)=0}$

The only Newton polygon for which $e=1$ is represented in Figure \ref{prang1}.
\begin{figure}[h]
   \centering
   \begin{pspicture}(0,0)(6,3)
   \psset{unit=7mm,labels=none}
  \rput(-0.4,-0.4){$0$}
  \rput(-0.5,1){$n$}
  \rput(-0.5,2){$2n$}
  \rput(-0.5,3){$3n$}
  \rput(1,-0.4){$1$}
    \rput(2,-0.4){$2$}
    \rput(3,-0.4){$3$}
   \rput(4,-0.4){$4$}
    \rput(5,-0.4){$5$}
    \rput(6,-0.4){$6$}
   \psaxes[linewidth=.7\pslinewidth]{-}(0,0)(0,0)(6,3)
  \psgrid[griddots=5, subgriddiv=0, gridlabels=0pt](0,0)(6,3)
   \psline(0,3)(1,2)(5,0)(6,0)
\end{pspicture}
 \caption{\label{prang1}$p$-rank $1$ case}
\end{figure}
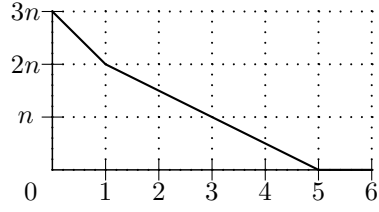

This is the Newton polygon of  $p(t)$ if and only if $v_p(a_2)\geq n/2$ and $v_p(a_3)\geq n$. If these conditions hold, $p(t)$ has a factor  in $\mathbb{Q}_p$ of degree $4$ with roots of valuation $n/2$ and thus $e=1$ if and only if this factor has no root in $\mathbb{Q}_p$, that is, if and only if $p(t)$ has no root of valuation $n/2$ in $\mathbb{Q}_p$.

\bigskip
\noindent
\textbf{p-rank 0 case: }$\mathbf{v_p(a_3)>0}$\textbf{, }$\mathbf{v_p(a_2)>0}$\textbf{ and }$\mathbf{v_p(a_1)>0}$

There are two Newton polygons for which $e=1$. One is represented in Figure \ref{type13}.
\begin{figure}[h]
   \centering
    \begin{pspicture}(0,0)(6,3)
     \psset{unit=7mm,labels=none}
  \rput(-0.4,-0.4){$0$}
  \rput(-0.5,1){$n$}
  \rput(-0.5,2){$2n$}
  \rput(-0.5,3){$3n$}
  \rput(1,-0.4){$1$}
    \rput(2,-0.4){$2$}
    \rput(3,-0.4){$3$}
   \rput(4,-0.4){$4$}
    \rput(5,-0.4){$5$}
    \rput(6,-0.4){$6$}
   \psaxes[linewidth=.7\pslinewidth]{-}(0,0)(0,0)(6,3)
  \psgrid[griddots=5, subgriddiv=0, gridlabels=0pt](0,0)(6,3)
   \psline(0,3)(3,1)(6,0)
\end{pspicture}
 \caption{\label{type13}Type $1/3$ case}
\end{figure}
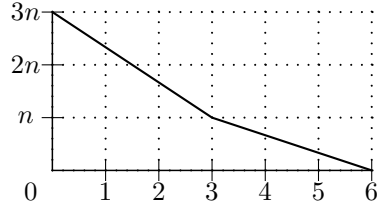

This is the Newton polygon of  $p(t)$ if and only if $v_p(a_1)\geq n/3$, $v_p(a_2)\geq 2n/3$ and $v_p(a_3)=n$.  If these conditions hold, $p(t)$ has two factors  in $\mathbb{Q}_p$ of degree $3$, one with roots of valuation $2n/3$ and the other with roots of valuation $n/3$;  $e=1$ if and only if those factors are irreducible in $\mathbb{Q}_p$, that is, if and only if $p(t)$ has no root in $\mathbb{Q}_p$.

\medskip

The other Newton polygon is represented in Figure \ref{super}; the corresponding abelian varieties are supersingular.
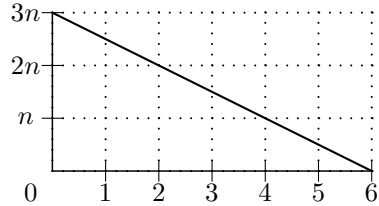
\begin{figure}[h]
  \centering
    \begin{pspicture}(0,0)(6,3)
     \psset{unit=7mm,labels=none}
  \rput(-0.4,-0.4){$0$}
  \rput(-0.5,1){$n$}
  \rput(-0.5,2){$2n$}
  \rput(-0.5,3){$3n$}
  \rput(1,-0.4){$1$}
    \rput(2,-0.4){$2$}
    \rput(3,-0.4){$3$}
   \rput(4,-0.4){$4$}
    \rput(5,-0.4){$5$}
    \rput(6,-0.4){$6$}
   \psaxes[linewidth=.7\pslinewidth]{-}(0,0)(0,0)(6,3)
  \psgrid[griddots=5, subgriddiv=0, gridlabels=0pt](0,0)(6,3)
   \psline(0,3)(6,0)
\end{pspicture}
 \caption{\label{super}Supersingular case}
\end{figure}

This is the Newton polygon of  $p(t)$ if and only if $v_p(a_1)\geq n/2$, $v_p(a_2)\geq n$ and $v_p(a_3)\geq 3n/2$.  If these conditions hold, $e=1$ if and only if $p(t)$ has no root nor factor of degree $3$ in $\mathbb{Q}_p$.


\section{Supersingular case}
Nart and Ritzenthaler \cite{nr} proved that the only supersingular $q$-Weil numbers of degree six are
$$\begin{array}{ll}
\pm\sqrt{q}\zeta_7,\quad\pm\sqrt{q}\zeta_9,                                 &    \quad\mbox{if }q\mbox{ is a square,} \\
\sqrt{q}\zeta_{28}\:(p=7),\quad \sqrt{q}\zeta_{36}\:(p=3),            &    \quad\mbox{if }q\mbox{ is not a square,} 
\end{array}$$
where $\zeta_n$ is a primitive $n$th root of unity.

We will use this result to obtain a list of possible supersingular characteristic polynomials as stated in Proposition \ref{supers}. We will have to calculate the minimal polynomial of some algebraic integers; in order to do this, we will often use the (trivial) fact that if $\alpha$ is a root of $f(t)=\sum_{i=0}^nb_it^{n-i}$ and $a\in\mathbb{C}$ then $a\alpha$ is a root of $f_a(t)=\sum_{i=0}^nb_ia^it^{n-i}$. We denote by $\phi_n(t)$ the $n$th cyclotomic polynomial.

\bigskip
\noindent
$\bullet$ If $q$ is a square, as $\phi_7(t)=t^6+t^5+t^4+t^3+t^2+t+1$,  the minimal polynomial of $\sqrt{q}\zeta_7$ (respectively $-\sqrt{q}\zeta_7$) is $p(t)=t^6+q^{1/2}t^5+qt^4+q^{3/2}t^3+q^2t^2+q^{5/2}t+q^3$ (resp. $p(t)=t^6-q^{1/2}t^5+qt^4-q^{3/2}t^3+q^2t^2-q^{5/2}t+q^3$). If $p\neq 7$, $p(t)$ has no factor of degree $1$ and $3$ over $\mathbb{Q}_p$ if and only if $\mathbb{Q}_p$ and its cubic extensions do not contain a $7$th primitive root of unity; this is equivalent  (see \cite[Proposition 2.4.1., p.53]{ami}) to
$$7\not\vert (p^3-1).$$

In the same way, $\phi_9(t)=t^6+t^3+1$ and the minimal polynomial of $\sqrt{q}\zeta_9$ (respectively  $-\sqrt{q}\zeta_9$) is $p(t)=t^6+q^{3/2}t^3+q^3$ (resp. $p(t)=t^6-q^{3/2}t^3+q^3$). If $p\neq 3$, $p(t)$ have no factor of degree $1$ and $3$ over $\mathbb{Q}_p$ if and only if 
$$3\not\vert (p-1).$$

If $p=7$ in the first case or $p=3$ in the second case, $p(t)$ is irreducible over $\mathbb{Q}_p$ (apply Eisenstein's Criterion to $p(t+1)$).

\bigskip
\noindent
$\bullet$ Suppose that $q$ is not a square. When $p=7$, as $\phi_{28}(t)=t^{12}-t^{10}+t^8-t^6+t^4-t^2+1$, the monic polynomial with roots $\sqrt{q}\zeta_{28}$ is $t^{12}-qt^{10}+q^2t^8-q^3t^6+q^4t^4-q^5t^2+q^6$ which is the product of
$$t^6+\sqrt{pq}t^5+3qt^4+q\sqrt{pq}t^3+3q^2t^2+q^2\sqrt{pq}t+q^3$$
and
$$t^6-\sqrt{pq}t^5+3qt^4-q\sqrt{pq}t^3+3q^2t^2-q^2\sqrt{pq}t+q^3.$$

 When $p=3$, as $\phi_{36}(t)=t^{12}-t^6+1$, the monic polynomial with roots $\sqrt{q}\zeta_{36}$ is $t^{12}-q^3t^6+q^6$ which is the product of
$$t^6+q\sqrt{pq}t^3+q^3$$
and
$$t^6-q\sqrt{pq}t^3+q^3.$$

The resulting polynomials are characteristic polynomials of abelian varieties of dimension $3$ (see \cite{nr}).


\section*{Acknowledgements}
I would like to thank Christophe Ritzenthaler for suggesting this work and for helpful discussions. I would also like to thank my thesis advisor, Yves Aubry, for his helpful comments and Hamish Ivey-Law for his careful reading of this paper.


\end{document}